\documentclass{amsart}[12pt]
\usepackage{amsmath,amsfonts,amssymb,latexsym,amsthm}
\usepackage{epsfig}
\usepackage{hyperref} %hyperref wants to be loaded last
\usepackage[dvipsnames]{xcolor}

\numberwithin{equation}{section}

\newtheorem{thm}{Theorem}[section]
\newtheorem{lemma}[thm]{Lemma}

\newtheorem{prop}[thm]{Proposition}
\newtheorem{conj}[thm]{Conjecture}
\newtheorem{conv}[thm]{Convention}

\newcommand{\ns}{\normalsize}

\title[Concrete polytopes may not tile the space]
{Concrete polytopes may not tile the space}

\author[Alexey Garber and Igor Pak]{Alexey Garber$^\star$ and Igor Pak$^\diamond$}

\thanks{\today}

\thanks{\thinspace ${\hspace{-.45ex}}^\star$School of Mathematical \& Statistical Sciences,
University of Texas Rio Grande Valley, Brownsville, TX, 78520}

\thanks{\thinspace {\hspace{.45ex}}
Email:
\hskip.06cm
\texttt{alexeygarber@gmail.com}}

\thanks{\thinspace ${\hspace{-.45ex}}^\diamond$Department of Mathematics,
UCLA, Los Angeles, CA~90095.
\hskip.06cm
Email:
\hskip.06cm
\texttt{pak@math.ucla.edu}}

\def\zz{\mathbb Z}
\def\nn{\mathbb N}

\def\rr{\mathbb R}
\def\qqq{\mathbb Q}

\def\ov{\overline}

\def\Ga{\Gamma}

\def\de{\delta}

\def\al{\alpha}

\def\om{\omega}

\def\ve{\varepsilon}

\def\cP{\mathcal P}

\def\ssu{\subset}

\def\<{\langle}
\def\>{\rangle}

\def\0{{\mathbf 0}}

\def\SS{{\mathbb{S}}}
\def\DD{{\mathbb{D}}}

\def\conv{{\textrm{conv}}}

\def\.{\hskip.06cm}
\def\ts{\hskip.03cm}

\DeclareMathOperator{\vol}{vol}
\begin{document}

%\thanks{The first author has been supported by the Humboldt Foundation}

%i\newcommand{\R}{\ensuremath{\mathbb{R}}}
%\newcommand{\Z}{\ensuremath{\mathbb{Z}}}
%\newcommand{\NN}{\ensuremath{\mathbb{N}}}
%\newcommand{\QQ}{\ensuremath{\mathbb{Q}}}
%\newcommand{\CC}{\ensuremath{\mathbb{C}}}
%\newcommand{\ZZ}{\ensuremath{\mathbb{Z}}}
%\newcommand{\TT}{\ensuremath{\mathbb{T}}}

%\theoremstyle{definition}
%\newtheorem{definition}[theorem]{Definition}

%\theoremstyle{remark}
%\newtheorem*{remark}{Remark}
%\newcommand{\end{proof}}{\hfill $\Box$ }
%\newcommand{\begin{proof}}{\noindent{\bf Proof.}\ \ }

\begin{abstract}
Brandolini et al. conjectured in~\cite{BCRT}
that all \emph{concrete lattice polytopes} can multitile the space.
We disprove this conjecture in a strong form, by constructing
an infinite family of counterexamples in~$\rr^3$.
\end{abstract}

\maketitle

\section{Introduction}\label{s:into}
The study of \emph{integer points in convex polytopes} is so challenging
because it combines the analytic difficulty of number theory with
hardness of imagination typical to high dimensional geometry and the
computational complexity of integer programming.  Consequently,
whenever a new conjecture is posed it is a joyful occasion, as it suggests
an order in an otherwise disordered universe.  When a conjecture
is occasionally disproved, this adds another layer of mystery
to the subject.

% \smallskip

In this paper we study the conjecture by Brandolini et~al.~\cite[Conj.~5]{BCRT}
that all concrete lattice polytopes can multitile the space. This
conjecture was restated and further investigated in~\cite[Conj.~8.6]{MR}
from a different point of view.
Here we disprove the conjecture by constructing a series
of explicit counterexamples. In fact, our main result is stronger
as it holds under more general notion of tileability.  Our tools
involve McMullen's theory of \emph{valuations of lattice polyhedra}
and \emph{Dehn's invariant}.  We conclude with final remarks and open problems.

\smallskip

A convex polytope $P \ssu \rr^d$ is called a \emph{lattice polytope}
if all its vertices are in $\zz^d$. Denote by
$$\om_P(x) \. := \frac{\vol\bigl(B_\ve(x)\cap P\bigr)}{\vol B_\ve(x)}
$$
the \emph{solid angle} at point~$x$, where $B_\ve(x)$ is a ball
of radius~$\ve$ centered at~$x$, and $\ve>0$ sufficiently small.

Define the \emph{$($regularised$)$ discrete volume} of $P$
as the sum of solid angles over all integer points:
$$
\chi(P) \. := \. \sum_{x\in P\cap \zz^d} \. \om_P(x)\ts,
$$
cf.~\cite{Bar,BR}.  \emph{Pick's theorem}
says that $\chi(P)=\vol(P)$ for all lattice polygons $P\ssu \rr^2$.
In an attempt to extend the theorem, Brandolini et~al.~\cite{BCRT} call
a lattice polytope \emph{concrete} if $\chi(P)=\vol(P)$.
They made the following curious conjecture.

\smallskip

\begin{conj}[{\cite{BCRT,MR}}]
Every concrete lattice polytope $P\ssu \rr^d$ \emph{multitiles}~$\rr^d$
with parallel translations and finitely many reflections.
\label{conj:BCRT}
\end{conj}

\smallskip

Here we say that $P$ \emph{multitiles}~$\rr^d$ if there is an integer
$k\ge 1$ and an
infinite family $\cP$ of congruent copies of~$P$ such that every
generic point $x\in \rr^d$ belongs to exactly $k$ polytopes in~$\cP$,
see~\cite{GRS}.
For $k=1$ this is the usual tiling of the space, see e.g.~\cite{GS}.
In the conjecture, only $P'\in \cP$ are allowed if they are obtained
from $P$ with parallel translations and finitely many reflections.
We disprove a stronger claim in the main result of this paper:

\smallskip

\begin{thm} There is a concrete lattice polytope $P\ssu \rr^3$
which does not multitile $\rr^3$.  Moreover, for all~$N$,
there is such a polytope $P$ with more than $N$ vertices.
\label{t:main}
\end{thm}

\medskip

\section{Brief background and the countexample idea}\label{s:back}

Let us first expound on the background and the motivation
behind the conjecture.  The problem of classifying polytopes
which can tile (tesselate) the space is classical.  It goes
back to the works of F\"edorov, Minkowski, Voronoi, Delone and Alexandrov,
and was featured in \emph{Hilbert's 18-th Problem}, see~\cite{GS}.
For tilings with parallel translations much more is known;
notably that in $\rr^3$ all such polytopes must be
\emph{zonotopes} (polytopes with centrally-symmetric faces of
all dimensions).  In higher dimensions, or for larger discrete
groups of translations and reflections, other polytopes appear
to tile the space, e.g.\ the \emph{24-cell} in~$\rr^4$.

For the lattice polytopes, the tilings are also heavily constrained and
can be studied using analytic tools~\cite{BCRT,GRS}.  The notion
of \emph{multitiling} goes back to Furtw\"angler (1936), and
many classical tiling results extend to this setting~\cite{GKRS}.
It is known and easy to see that if a lattice polytope multitiles
the space with parallel translations then it is
concrete~\cite{BCRT} (see below).  In particular,
all lattice zonotopes multitile the space~\cite{GRS},
and they are concrete because they can be partitioned
into parallelepipeds, see e.g.~\cite[$\S$7]{BP} and~\cite[Ch.~7]{Zie}.
The conjecture can then be viewed as an attempt to say that
the class of concrete lattice polytopes is very small and
can be characterized via the large body of work towards
characterization of tilings and multitilings.

From this point on, we restrict ourselves to convex polytopes
$P\ssu \rr^3$. For the clarity, observe that $\om_P(x)= 1$ for $x$
in relative interior of~$P$, and $\om_P(x)=\frac12$ for $x$
in relative interior of a face. Similarly,
$\om_P(x)=\al(e)$ for $x$ in relative interior of an edge~$e$ of~$P$,
where $\al(e)$ is the dihedral angle at~$e$, and $\om_P(x)$
is the usual solid angle for a vertex~$x$ of~$P$.

There is a way to understand both the conjecture and
our theorem as part of the same asymptotic argument.
For a polytope $P\ssu \rr^3$, define the (\emph{lattice})
\emph{volume defect} by
$$
\de(P) \. :=\. \chi(P) \. - \. \vol(P)\, \in \rr\ts.
$$
Similarly, the \emph{Dehn invariant} is given by
$$
\DD(P) \. := \. \sum_{e  \in E(P)} \. \ell(e) \otimes \al(e) \, \in \, \rr \. \otimes_{\zz} \. \bigl(\rr/\pi\zz\bigr)\ts,
$$
where $E(P)$ is the set of edges in~$P$, $\ell(e)$ is the length of edge~$e$,
and $\al(e)$ is the dihedral angle at~$e$, see e.g.~\cite{Bol,Dup}.

\smallskip

\begin{thm}% [{see below}]
Let $P\ssu \rr^3$ be a convex polytope which multitiles
the space.  Then \ts $\DD(P)=0$.  Similarly, let $P$ be a
lattice convex polytope which multitiles the space
with parallel translations. Then \ts $\de(P)=0$, i.e.~$P$ is concrete.
\label{t:multitile}
\end{thm}

\smallskip

Versions of this result and its various generalizations have been
repeatedly rediscovered, often with the same asymptotic argument
which goes back to Debrunner (1980) and M\"urner (1975).  We refer
to~\cite{LM} for generalizations to higher dimensions and further
references (see also~$\S$\ref{ss:finrem-asy}).

\smallskip

\begin{proof}[Proof outline] First, suppose the multitiling is the
usual tiling. Let $\cP$ be the set of copies of $P$
which define the usual tiling of~$\rr^3$, and let $\cP_R\ssu \cP$
be the set of copies of $P$ which intersect a ball $B_R(O)$
of radius~$R$ around the origin. Denote by $\Ga\ssu \rr^3$ the
region covered with tiles in~$\cP_R$. On the one hand, both
the volume defect and the Dehn invariant are additive, so
\ts $\de(\Ga) = |\cP_R| \.\de(P) = \Theta(R^3)\.\de(P)$, and \ts
$\DD(\Ga) = |\cP_R| \.\DD(P) = \Theta(R^3)\.\DD(P)$.  On the other hand,
both the volume defect and the Dehn invariant depend only on the boundary
$\partial\ts\Ga$, which is within a constant distance from $\partial B_R(O)$.
Thus both grow at most quadratically: \ts $\de(\Ga)$,
$\DD(P) = O(R^2)$.  For the defect this is clear, and for the Dehn invariant 
this can be made precise by using Kagan functions~$f$ (see below), 
which extend to ring homomorphisms \ts 
$\rr \. \otimes_{\zz} \. \bigl(\rr/\pi\zz\bigr) \to \rr$, 
so the resulting function of~$R$ can then be viewed analytically. 
Comparing the lower and upper asymptotic bounds, this implies 
the result.  \ts For general
$k$-tilings, the same argument work verbatim, as the changes 
are straightforward.\footnote{We are
implicitly using the fact that $\de(P)=\de(P')$ for all copies of $P'\in \cP$.
This is not true when reflections are allowed, see~$\S$\ref{ss:finrem-ortho}.}
\end{proof}

\smallskip

Now, the idea of a counterexample to Conjecture~\ref{conj:BCRT} is very clear.
We use the lattice valuation theory to construct a lattice polytope
$P\ssu \rr^3$ whose volume defect \ts $\de(P)=0$, while the Dehn invariant
$\DD(P)\ne 0$.  By Theorem~\ref{t:multitile}, this implies that $P$ cannot
multitile the space.

\bigskip

\section{Minkowski additivity}\label{s:Mink}

\subsection{Volume defect}
By definition, the volume defect $\de(P)$ is a translation invariant valuation
on lattice polytopes, i.e. $\de(P+x)=\de(P)$ \ts for all $x\in \zz^3$, and
$$
\de(P\cup Q) + \de(P\cap Q) \, = \, \de(P)+\de(Q)\ts,
$$
for all lattice polytopes $P, Q\ssu \rr^3$.  In particular, the volume defect
is additive under disjoint union (except at the boundary).  We also
need the following linearity property under \emph{Minkowski addition} \ts
$P+Q = \{x+y \. |\. x\in P, y\in Q\}$ \ts and \emph{expansion} \ts
$c\ts P = \{c\ts x\. |\. x\in P\}$.

\begin{lemma}[see $\S$\ref{ss:finrem-lemma}]
Let $P_1,\ldots,P_k$ be lattice convex polytopes in~$\rr^3$, and
let $t_1,\ldots,t_k \in \nn$.
Then
$$
\de\bigl(t_1 P_1 \.+\. \ldots \.+\. t_kP_k\bigr) \, = \,
t_1\ts \de(P_1) \.+\. \ldots \.+\. t_k \ts \de(P_k)\ts.
$$
\label{l:defect}
\end{lemma}

\begin{proof}[Proof outline]
Let $P\ssu \rr^3$ be a lattice polytope, and let $t\in \nn$
be a variable.  Both $\chi(tP)$ and
$\vol(tP)$ are cubic polynomials, see~\cite[p.~127]{McMullen} (see also
\cite[Thm~7.9]{BL} and \cite[Thm~2.1]{Jochemko} for surveys
and further references).  Moreover, $\chi(tP)$ and
$\vol(tP)$ are \emph{odd} cubic polynomials, see~\cite[Thm~4.8]{Macdonald}
(see also~\cite{McM-congress}), with the same leading coefficient.
Thus, $\de(P)=\chi(P)-\vol(P)$ is linear.  The polynomiality under
Minkowski addition follows from McMullen's
\emph{homogeneous decomposition}~\cite{McM-congress}
(see also e.g.~\cite[Thm~4.1]{Jochemko}).
Again, the cubic terms cancel, and the same argument proves
multilinearity as in the lemma.
\end{proof}

\subsection{Dehn invariant}
For the clarity of exposition, we follow~\cite[$\S$17]{Pak} (see also~\cite{Bol,Dup}).
Fix an \emph{additive function} \ts $f:\rr\longrightarrow \rr$, s.t.
$f(a+b)=f(a)+f(b)$ for all $a,b\in \rr$. Additive function~$f$ s.t. $f(\pi)=0$ \ts
is called a \emph{Kagan function}, after~\cite{Kag}.

For a convex polytope $P \ssu \rr^3$ and a Kagan function~$f$, denote
$$
D_f(P) \. := \. \sum_{e\in E(P)} \. \ell(e) \ts f(\alpha_e)\ts.
$$
Observe that $D_f$ is a translation invariant valuation,
and that \ts $D_f(cP)=cD_f(P)$.

\begin{lemma}[see $\S$\ref{ss:finrem-lemma}]
Let $P_1,\ldots,P_k$ be convex polytopes in~$\rr^3$, and
let $t_1,\ldots,t_k \in \rr_+$.
Then:
$$
D_f\bigl(t_1 P_1 \.+\. \ldots \.+\. t_kP_k\bigr) \, = \,
t_1\ts D_f(P_1) \.+\. \ldots \.+\. t_k \ts D_f(P_k)\ts,
$$
for every Kagan function~$f$.
\label{l:Dehn}
\end{lemma}

\begin{proof}[Proof outline]
For $k=2$, the result follows immediately
from the homogeneous decomposition again and the additivity
of the Dehn invariant under disjoint union. For larger~$k$,
proceed by induction.  

Alternatively, recall that the 
Dehn invariant is a simple, translation-invariant valuation 
(see $\S$\ref{ss:finrem-lemma}), so \ts 
$D_f\bigl(t_1 P_1 \.+\. \ldots \.+\. t_kP_k\bigr)$ \ts 
is a polynomial in \ts $t_1,\ldots, t_k$ \ts of degree 
at most~3. Now, the restriction of that polynomial onto 
any ray from the origin is a linear function, which 
implies that this polynomial is linear. 
The details of both arguments are straightforward.
\end{proof}

\medskip

\section{Counterexample construction}\label{s:counterexample}

Consider the following three tetrahedra:
{\small
$$\aligned
T_1 \. & :=\. \conv \bigl\{(0,0,0),\. (1,1,0),\. (1,0,1),\. (0,1,1) \bigr\},\\
T_2\. & :=\. \conv \bigl\{(0,0,0),\. (2,2,-1),\. (2,-1,2),\. (1,-2,-2) \bigr\},\\
T_3\. & :=\. \conv \bigl\{(0,0,0),\. (2,2,-1), \. (3,0,-3),\. (5,-1,-1) \bigr\}.
\endaligned
$$}

%\nin
In notation of~\cite[$\S$16]{Pak}, tetrahedron~$T_1$ is \emph{regular}
with edge length~$\sqrt 2$, tetrahedron~$T_2$ is \emph{standard} with
three pairwise orthogonal edges of length~$3$, and $T_3$ is an \emph{orthoscheme}
(also called \emph{path simplex} and \emph{Hill tetrahedron}),
with three edge lengths~$3$. 

Note that six copies of~$T_3$ tile a cube
spanned by vectors \ts $v_1=(2,2,-1)$, \ts $v_2=(1,-2,-2)$, and 
$v_3=(2,-1,2)$ starting at the origin~$O$.  Indeed, these six 
copies correspond to six permutations of \ts $\{v_1,v_2,v_3\}$, 
and are spanned by the paths formed by these vectors.  This 
implies that $D_f(T_3)=0$ for every Kagan function~$f$ defined 
above.

\begin{prop}
Let \ts $P:= \ts 5 \ts T_1 \ts +\ts 12 \ts T_2\ts +\ts 19 \ts T_3$.
Then \ts $\de(P)=0$, and \ts $D_f(P)\ne 0$ \ts for some Kagan function~$f$.
\label{p:ex}
\end{prop}

\begin{proof}
Let \ts $\alpha=\arccos\frac13$\ts, and recall that $\frac{\al}{\pi}\notin \qqq$,
see e.g.~\cite[$\S$41.3]{Pak} and~\cite{Bol}. Thus, there is a Kagan function~$f$
which satisfies \ts $f(\al)\ne 0$, and, moreover, \ts $f(\al) \notin \ov\qqq$,
see~\cite[Ex.~17.8]{Pak}.

Observe that all dihedral angles of $T_1$ are equal to~$\alpha$.
Dihedral angles of $T_2$ are equal to $\frac{\pi}{2}$ for the
three edges at the origin, and to \ts
$\beta:=\arccos\frac{\sqrt 3}{3}=\frac{\pi -\alpha}{2}$ \ts
for the three other edges. Finally, all dihedral angles of
$T_3$ are rational multiples of~$\pi$.
The values of the volume defect and the Dehn invariant
for all three tetrahedra are given in Table~\ref{tab:values} below.

{\Large
\begin{table}[hbt]
\begin{center}
\begin{tabular}{|r|c|c|c|}
\hline
& {\large $T_1$} & {\large $T_2$} & {\large $T_3$} \\ \hline
{\large $\de(\cdot)$} & {\ns ${
%\displaystyle
 \frac{3\alpha}{\pi}-\frac43}$ } & {\ns ${
%\displaystyle
-\frac{5\alpha}{4\pi}-\frac12}$ } & {\ns \ts $\frac23$} \\%\hline
{\large \, $D_f(\cdot)$} & {\small $6\sqrt 2$}{\ns \ts $f(\alpha)$ } & {\ns ${%\displaystyle
-\frac{9}{\sqrt 2}\ts f(\alpha)}$ } &{\ns 0}\\\hline
\end{tabular}
\end{center}
\caption{Values of the volume defect and the Dehn invariant.}\label{tab:values}
\end{table}
}

Using values from the table, Lemmas~\ref{l:defect} and~\ref{l:Dehn} imply:
$$\aligned
\de(P) \, &= \, 5 \ts \de(T_1) \. + \. 12 \ts \de(T_2)\. + \. 19 \ts \de(T_3)\. = \. 0\ts,\\
D_f(P)\, & = \, 5 \ts D_f(T_1) \. + \. 12 \ts D_f(T_2)\. + \. 19 \ts D_f(T_3)\. = \. -24 \ts \sqrt 2 \ts f(\alpha) \. \ne \. 0\ts,
\endaligned
$$
as desired.\footnote{See~\cite[$\S$16.4]{Pak} for the
algebraic approach, which implies \ts $D_f(P)\ne 0$ \ts without computing
dihedral angles directly.}
\end{proof}

\begin{proof}[Proof of Theorem~\ref{t:main}]
The polytope $P\ssu \rr^3$ constructed in the proof of
Proposition~\ref{p:ex} is concrete, but has a non-zero Dehn invariant.
Thus, by Theorem~\ref{t:multitile}, it cannot multitile the space.
This proves the first part of the theorem.

For the second part, take a lattice zonotope $Q\ssu \rr^3$ with
at least $N$ vertices.  From the results in~$\S$\ref{s:back}, we have
$\de(Q)=0$, and $D_f(Q)=0$ for all Kagan functions~$f$. By Lemma~\ref{l:defect},
we have \ts $\de(P+Q)=\de(P)=0$, so $(P+Q)$ is concrete.  On the other hand,
by Lemma~\ref{l:Dehn}, we have \ts $D_f(P+Q)=D_f(P)\ne 0$. Thus, by
Theorem~\ref{t:multitile}, polytope $(P+Q)$ cannot multitile~$\rr^3$.
Finally, observe that $(P+Q)$ has at least $N$ vertices,
see e.g.~\cite[Prop.~7.12]{Zie}. This completes the proof.
\end{proof}

\medskip

\section{Final remarks and open problems}

\subsection{} Curiously, the volume defect can be both very
small or very large for general lattice polytopes in $\rr^3$.
Indeed, consider the following \emph{wedge tetrahedron}
and \emph{flat square pyramid}:
{\small
$$\aligned
W_n \. & :=\. \conv \bigl\{(0,0,0),\. (1,1,0),\. (1,0,n),\. (0,1,n) \bigr\},\\
V_n \. & :=\. \conv \bigl\{(0,0,0),\. (n,0,0),\. (0,n,0),\. (n,n,0),\. (0,0,1) \bigr\}.
\endaligned
$$
}
As $n\to \infty$, we have:
$$\chi(W_n) \. = \. \Theta\left(\frac{1}{n}\right)\,, \qquad
\vol(W_n) \. = \. \frac{n}{3}\,, \qquad \text{and} \qquad
\de(W_n) \. \sim \. -\frac{n}{3}\,.
$$
On the other hand,
$$
\chi(V_n) \. = \. \frac{n^2}{2} \. - \. O(n)\,, \qquad
\vol(V_n) \. = \. \frac{n^2}{3}\,, \qquad \text{and} \qquad
\de(V_n) \. \sim \. \frac{n^2}{6}\,.
$$

\subsection{} \label{ss:finrem-lemma}
Lemma~\ref{l:defect} follows from a more general
result in the literature that every translation invariant
valuation on $\rr^d$ which is homogeneous of degree one is Minkowski
additive (see~\cite[Rem.~6.3.3]{Sch} and~\cite[Cor.~32]{BL2}).
We include a short proof outline both for simplicity, to remain
as much self-contained as possible, and as a brief guide to
the literature.  While Lemma~\ref{l:Dehn} is very natural,
we could not find it stated in this form.  As we explain above,
its proof follows along steps similar to the proof of
Lemma~\ref{l:defect}.

\subsection{} \label{ss:finrem-asy}
The asymptotic argument in the proof of Theorem~\ref{t:multitile}
can also be applied in $\rr^2$, where it is traditionally used to show that
the plane cannot be tiled with congruent convex $n$-gons, for $n\ge 7$.
See~\cite{Ale,Niv} for early versions of this result.  See also~\cite[Thm~D]{KPP}
for an advanced version of this argument, proving that strictly acute tetrahedra
cannot tile~$\rr^4$, and for further references.

\subsection{}
One can ask if the results of this paper can be further extended.
First, we can always extend Theorem~\ref{t:main} to higher
dimensions $d\ge 4$.  By the argument in the proof of
Theorem~\ref{t:multitile}, every $P\ssu \rr^d$ which multitiles~$\rr^d$
has zero \emph{Hadwiger} (generalized Dehn) {\em invariants}~\cite{LM}.
Take an orthogonal prism $P \times [0,1]$ over the polytope $P\ssu \rr^3$ as in
Proposition~\ref{p:ex}.  The dihedral angles are either $\pi/2$ or the same
as in~$P$. Thus the corresponding codim-2 Hadwiger invariant is non-zero,
giving a counterexample in~$\rr^4$.  Proceed by induction; the details
are straightforward.

Going one step further, we say that a lattice polytope $P\ssu\rr^d$ is
\emph{super concrete}, if it is concrete and \emph{scissors congruent}
to a $d$-cube.  For \ts $d=3,\ts 4$, by the \emph{Sydler--Jessen theorem}
this is equivalent to zero Dehn invariant~\cite{Bol,Dup}.  For $d\ge 5$,
scissors congruence with a $d$-cube implies and is conjecturally equivalent
to zero Hadwiger invariants, see e.g.~\cite{Zak}.  Also, every $P\ssu \rr^d$
which multitiles~$\rr^d$ with translations must be super concrete, see~\cite{LM}.
So in the spirit of Conjecture~\ref{conj:BCRT}, one can ask whether
 all super concrete lattice polytopes $P\ssu \rr^d$ can multitile the space.

We conjecture that the answer is negative for all $d\ge 3$.
For example, for $d\ge 4$, the local structure of cones around
a vertex can be constrained by a \emph{spherical Dehn invariant},
see e.g.~\cite{Dup}.  In principle, the concrete assumption is too weak
and can allow ``bad cones'' which would locally not multitile the
sphere $\SS^{d-1}$.  It would be interesting to make this precise.
The above problem is even more interesting in~$\rr^3$. In principle,
the cones around vertices can all have nontrivial geometry generating
a non-discrete group of symmetries, cf.~\cite{MM}.  Again, it would be
interesting to give an explicit construction.

\subsection{}\label{ss:finrem-ortho}
Let us mention that the proof of Proposition~\ref{p:ex} hinges on the following
curious geometric property: the orthoscheme $T_3$ tiles the lattice
cube and thus the space, yet has $\de(T_3)\ne 0$.  In particular, this shows that
Theorem~\ref{t:multitile} cannot be extended to allow reflections.
This non-zero volume defect of $T_3$ has to do with the fact that the remaining
five orthoschemes in the tiling of the cube are obtained from~$T_3$ by
reflections which do not preserve the lattice.
Although congruent to~$T_3$, these reflected orthoschemes have both negative
and positive volume defect, giving zero in total for the lattice cube.

Note that the (primitive) lattice cubes which arise in the construction,
correspond (up to parallel translation) to rational orthogonal matrices
$M\in O(3,\rr)$.  These matrices are enumerated in~\cite{Cre}.  We conjecture
that the corresponding orthoschemes have a nonzero volume defect with probability
at least \ts $\ve>0$, as the cube edge length $\ell\to \infty$.
This would give further examples of polytopes as in
the proposition, all with a bounded number of vertices.

\vskip.6cm

\subsection*{Acknowledgments.}
The authors are grateful to Sasha Barvinok, Alexey Glazyrin
and Sinai Robins for interesting remarks.  Special thanks to
Katharina Jochemko and Monika Ludwig for help with the
valuation literature.  The paper was finished when the
second author was on sabbatical at the Mittag-Leffler
Institute; we are grateful for the hospitality.
The second author was partially supported by the NSF.

\vskip1.1cm

% \newpage

%\appendix

%\section{Computations of $\chi^o$ and $D_f$}\label{appendix}

\end{document}